\documentclass[11pt]{amsart}
\usepackage{epsfig}
\usepackage{amsbsy,latexsym}
\usepackage{amsmath}
\usepackage{amssymb, mathrsfs}
\usepackage[mathscr]{eucal}
\usepackage{hyperref}
\usepackage{amsfonts}
\usepackage{amsthm}
\usepackage{amsmath}
\usepackage{amsfonts}
\usepackage{latexsym}
\usepackage{amssymb}
\usepackage{amscd}
\usepackage[latin1]{inputenc}
\usepackage{verbatim}
\usepackage{tikz-cd}
\tikzset{
  symbol/.style={
    draw=none,
    every to/.append style={
      edge node={node [sloped, allow upside down, auto=false]{$#1$}}}
  }
}

\usepackage[english]{babel}

\newtheorem{definition}{Definition}[section]
\newtheorem{theorem}[definition]{Theorem}
\newtheorem{lemma}[definition]{Lemma}
\newtheorem{corollary}[definition]{Corollary}

\theoremstyle{definition}

\newcommand{\C}{\mathbb{C}}

\newcommand{\cl}[1]{\mathcal{#1}}

\newcommand\tr{ \operatorname{tr} }

\newcommand{\lm}{\lambda}
\newcommand{\la}{\langle}
\newcommand{\ra}{\rangle}
\newcommand{\Ltau}{\Lambda_{\tau}}

\title[Orthogonal Unitary Bases and a Subfactor Conjecture]{Orthogonal Unitary Bases and a Subfactor Conjecture}

\begin{document}

\author[J.~Crann, D.~W. Kribs, R. Pereira ]{Jason Crann$^{1}$, David~W.~Kribs$^{2}$, Rajesh Pereira$^{2}$}

\address{$^1$School of Mathematics \& Statistics, Carleton University, Ottawa, ON, Canada H1S 5B6}
\email{jasoncrann@cunet.carleton.ca}
\address{$^2$Department of Mathematics \& Statistics, University of Guelph, Guelph, ON, Canada N1G 2W1}
\email{dkribs@uoguelph.ca}
\email{pereirar@uoguelph.ca}




\begin{abstract}
We show that any finite dimensional von Neumann algebra admits an orthonormal unitary basis with respect to its standard trace. We also show that a finite dimensional von Neumann subalgebra of $M_n(\C)$ admits an orthonormal unitary basis under normalized matrix trace if and only if the normalized matrix trace and standard trace of the von Neumann subalgebra coincide. As an application, we verify a recent conjecture of Bakshi-Gupta \cite{bakshi2021few}, showing that any finite-index regular inclusion $N\subseteq M$ of $II_1$-factors admits an orthonormal unitary Pimsner-Popa basis.
\end{abstract}

\maketitle

\section{Introduction}

Orthonormal bases of unitaries arise naturally within operator algebras and their applications. In matrix algebras, such bases are known as unitary error bases in quantum information, and underly the structure of quantum teleportation protocols (see \cite{w01} and the references therein). In subfactor theory, orthonormal unitary Pimsner-Popa bases arise naturally for a large class of inclusions $N\subseteq M$ of $II_1$-factors (see, e.g., \cite{bg20,bakshi2021few,cs06}).

Motivated (in part) by a question of Popa \cite[\S 3.5]{p19}, Bakshi and Gupta recently established sufficient conditions under which a finite-index inclusion $N\subseteq M$ of $II_1$-factors admits an orthonormal unitary Pimsner-Popa basis \cite[Theorem 3.21]{bakshi2021few}. They conjectured that any finite-index \textit{regular} inclusion of $II_1$-factors admits such a basis \cite[Conjecture 3.20]{bakshi2021few}. In this note, we answer their conjecture in the affirmative.

Our approach largely follows that of Bakshi and Gupta, but our main contribution, which draws on techniques and tools from \cite{pereira2004trace}, gives a construction of an orthonormal basis (over $\mathbb{C}$) for any finite-dimensional von Neumann algebra relative to its standard/Markov trace.

\section{Main Result}

We define the normalized trace on the set of $d\times d$ complex matrices $M_d(\C)$ as $\tau(A)=\frac{1}{d}\tr(A)=\frac{1}{d}\sum_{i=1}^d a_{ii}$, where $A = (a_{ij})$.  A finite dimensional von Neumann algebra $M=\bigoplus_{i=1}^m I_{k_i}\otimes M_{n_i}(\C)\subseteq M_d(\C)$, where $d=\sum_{i=1}^mk_in_i$, inherits its normalized trace from $M_d(\C)$.  Since any unitary element is automatically norm one under the normalized trace inner product norm, we will often omit the normalization factor in our calculations since we only need verify the orthogonality of unitary elements.

We state our first result as follows:

\begin{theorem}\label{mainthm} Let $\{n_i\}_{i=1}^m$ be a set of natural numbers.  Let
\[
M=\bigoplus_{i=1}^m I_{n_i}\otimes M_{n_i}(\C).
\]
Then there exists a basis of $M$ consisting entirely of unitary matrices which are orthogonal under the normalized trace inner product.\end{theorem}

We note that the $m=1$ special case of this result is well known, see the proof of \cite[Proposition 3.23]{bakshi2021few} for instance.

In preparation for the proof below, we let $F_n$ be the $n \times n$ Fourier matrix (also called the DFT matrix in \cite{bakshi2021few}), given as follows:
\[
F_n = \frac{1}{\sqrt{n}}
\left[
\begin{matrix}
1 & 1 & \cdots & 1 \\
1 & \omega & \cdots & \omega^{n-1} \\
\vdots & \vdots & \vdots & \vdots \\
1  & \omega^{n-1} & \cdots & \omega^{(n-1)(n-1)}
\end{matrix}
\right],
\]
with $\omega=e^\frac{2\pi i}{n}$.
Fourier matrices are both unitary matrices and Vandermonde matrices, and they also appear in the spectral decomposition of circulant matrices \cite{davis2013circulant,kra2012circulant}.  (Recall that an $n \times n$ matrix $C$ is said to be circulant if  $c_{ij}=c_{kl}$ whenever $j-i=l-k\mod n$).    A matrix $C$ is circulant if and only if it is of the form $C=F_nDF_n^*$ for some diagonal matrix $D$ \cite[Theorem 5.8]{zhang2011matrix}.  From this, it follows that the circulant matrices form a $*$-algebra.  It also follows that any set of $n$ complex numbers of modulus one is the spectrum of an $n \times n$ unitary circulant matrix. Additionally, a feature that we shall use is that the main diagonal entries of a $n\times n$ circulant matrix $C$ are all equal to $\frac1n \tr(C)$.

With these facts in mind we can restate Theorem \ref{mainthm} in a more constructive fashion.

\begin{theorem}\label{mainthm2} Let $\{n_i\}_{i=1}^m$ be a set of natural numbers and let $d=\sum_{i=1}^m n_i^2$.  Let $M=\bigoplus_{i=1}^m I_{n_i}\otimes M_{n_i}(\mathbb{C})$.  Then there exists two unitaries $U,V\in M$ such that $\{ U^kV^k \}_{k=0}^{d-1}$ is an orthogonal basis of $M$ under the trace inner product. \end{theorem}

\begin{proof}
Let $d=\sum_{i=1}^m n_i^2$ and set $s_1=0$ and $s_i = \sum_{j=1}^{i-1} n_j^2$ for $1 < i \leq m$. Let $\omega=e^\frac{2\pi i}{d}$, and now define
\[
U= \bigoplus_{i=1}^m \omega^{s_i} I_{n_i}\otimes D_i ,
\]
where $D_i=\mathrm{diag}\,(1,\omega^{n_i}, \omega^{2n_i},....,\omega^{(n_i-1)n_i})$, and let
\[
V=\bigoplus_{i=1}^m I_{n_i}\otimes C_i ,
\]
where $C_i$ is a circulant matrix with eigenvalues $1,\omega, \omega^{2},....,\omega^{n_i-1}$.
Then
\[
\langle U^jV^j , U^kV^k\rangle= \tr(U^jV^j(U^kV^k)^*)=\tr(U^{j-k}V^{j-k}).
\]
We therefore need to show that the right hand side of this equation is zero whenever $j-k\not \equiv 0 \mod d$.

Let $r$ be any integer not divisible by $d$. Then observe that
$$
\tr(U^rV^r)=\sum_{i=1}^m \omega^{r s_i} n_i \tr(D_i^rC_i^r)=\sum_{i=1}^m \omega^{r s_i}\tr(D_i^r)\tr(C_i^r),
$$
where the last equality follows from the fact that $D_i^r$ is a diagonal matrix and every entry on the main diagonal of the circulant matrix $C_i^r$ is equal to $\frac{1}{n_i} tr(C_i^r)$. By choice of $r$, we have $\omega^r \neq 1$.
Hence we have
\begin{eqnarray*}
\tr(U^rV^r) &=& \sum_{i=1}^m \omega^{r s_i}(1+\omega^{n_1r}+ \omega^{2n_ir}+....+\omega^{(n_i-1)n_ir}) \\
& & \times (1+\omega^r+ \omega^{2r}+....+\omega^{(n_i-1)r}) \\
&=& \sum_{i=1}^m \omega^{r s_i} (1+\omega^{n_1r}+ \omega^{2n_ir}+....+\omega^{(n_i-1)n_ir}) \Big(  \frac{\omega^{n_i r} - 1}{\omega^{r} - 1}  \Big) \\
&=& \frac{1}{\omega^r -1} \sum_{i=1}^m \omega^{r s_i} (\omega^{n_i^2 r} - 1)  \\
&=&  \frac{\omega^{d} - 1}{\omega^{r} - 1} \\
&=& 0,
\end{eqnarray*}
and the result follows.
\end{proof}

We remark that the unitaries $U$ and $V$ are diagonal and block circulant respectively.  The matrices $U$ and $V$ from \cite[Proposition 3.23]{bakshi2021few} are diagonal and circulant respectively but the $U$ and $V$ in this construction are different in the $m=1$ special case from that in \cite{bakshi2021few}.

We can now characterize which finite dimensional von Neumann algebras have an orthonormal basis (under the normalized trace inner product) of unitary elements.  We need the following lemma, which is a special case of one direction of \cite[Corollary 1.2]{pereira2006representing}

\begin{lemma} \label{lemme} Let $M$ be a finite dimensional von Neumann algebra.  If $\{ x_i\}_{i=1}^{n}$ and $\{y_i\}_{i=1}^n$ are two orthonormal bases of $M$ under any fixed inner product on $M$,  then $\sum_{i=1}^n x_i^*x_i=\sum_{i=1}^n y_i^*y_i$.\end{lemma}

\begin{proof} 
If $\{ x_i\}_{i=1}^{n}$ and $\{y_i\}_{i=1}^n$ are two orthonormal bases of $M$, then $x_i=\sum_{j=1}^n \langle x_i, y_j\rangle y_j$ for all $i$, and similarly for the $y_i$.  Hence

\begin{eqnarray*} 
\sum_{i=1}^n x_i^*x_i &=& \sum_{i=1}^n \sum_{j,k=1}^n \overline{\langle x_i, y_j\rangle}\langle x_i, y_k\rangle  y_j^*y_k\\ 
&=& \sum_{j,k=1}^n \Big( \sum_{i=1}^n\overline{\langle x_i, y_j\rangle}\langle x_i, y_k\rangle \Big) y_j^*y_k\\ 
&=& \sum_{j,k=1}^n\langle y_j, y_k\rangle  y_j^*y_k \\ 
&=&\sum_{i=1}^ny_i^*y_i.
\end{eqnarray*}
\end{proof}

We now prove our result.

\begin{theorem}\label{mainthm3} Let  $\{k_i\}_{i=1}^m$  and $\{n_i\}_{i=1}^m$ be two sets of natural numbers.  Let
\[
M=\bigoplus_{i=1}^m I_{k_i}\otimes M_{n_i}(\C).
\]
Then there exists a basis of $M$ consisting entirely of unitary matrices that are orthogonal under the normalized trace inner product if and only if there exists a constant $c$ such that $n_i=c k_i$ for all $1\le i\le m$.\end{theorem}

\begin{proof} 
The `if' direction follows from Theorem \ref{mainthm}, which was proved by the construction in Theorem \ref{mainthm2}. Indeed, we can let $r=\mathrm{gcd}(k_1,k_2,...,k_m)$, so that $M=I_r \otimes M_1$ where $M_1$ is a von Neumann algebra, and $N=\oplus_{i=1}^m I_{ck_i} \otimes M_{n_i}(\C)$  is of the form $I_{cr} \otimes M_1$ (and where $cr$ is an integer following from the hypothesis). So there is a normalized trace preserving *-isomorphism between $M$ and $N$, and hence $M$ has a unitary orthonormal basis since $N$ does by the theorem above.

For the `only if' direction, denote by $\{e_{ij}^{(s)}\}_{i,j=1}^{n_s}$ the set of matrix units of $ M_{n_s}(\C)$. Then $\{ \{ k_s^{-\frac{1}{2}}I_{k_s}\otimes e_{ij}^{(s)}\}_{i,j=1}^{n_s}\}_{s=1}^m $ is orthonormal under the trace inner product and 
\[
\bigoplus_{s=1}^m k_s^{-1}I_{k_s}\otimes \Big( \sum_{i,j=1}^{n_s}e_{ji}^{(s)}e_{ij}^{(s)} \Big)=\bigoplus_{s=1}^m \frac{n_s}{k_s}I_{k_s}\otimes I_{n_s}.
\]   
Hence, if $M=\bigoplus_{i=1}^m I_{k_i}\otimes M_{n_i}(\C)$ has an orthonormal basis of unitary elements with respect to the normalized trace inner product, then by Lemma~\ref{lemme}, $\bigoplus_{i=1}^m \frac{n_i}{k_i}I_{k_i}\otimes I_{n_i}$ must be a multiple of the identity. This occurs only if there exists a constant $c$ such that $n_i=c k_i$ for all $1\le i\le m$.
\end{proof}

\section{Application}

In this section we verify \cite[Conjecture 3.20]{bakshi2021few} and mention a few corollaries. Beforehand, we review some preliminary notions from subfactor theory. See \cite{js97} for details.

Let $M$ be a tracial von Neumann algebra, that is, a finite von Neumann algebra equipped with a normal faithful tracial state $\tau$, and let $N\subseteq M$ be a (unital) von Neumann subalgebra. The inclusion $N\subseteq M$ is \textit{regular} if $\mathcal{N}_M(N)''=M$, where
$$\mathcal{N}_M(N)=\{u\in\mathcal{U}(M)\mid u^*Nu=N\}$$
is the unitary normalizer of $N$ in $M$.

The GNS construction of $(M,\tau)$ yields the Hilbert space $L^2(M,\tau)$, the GNS map $\Lambda_\tau:M\rightarrow L^2(M,\tau)$ and the (faithful) representation $\pi_\tau:M\rightarrow\cl B(L^2(M,\tau))$, where
$$\pi_\tau(x)\Lambda_\tau(y)=\Lambda_\tau(xy), \ \ \ x,y\in M.$$
Let $L^2(N,\tau)=\overline{\Ltau(N)}$ be the associated closed subspace of $L^2(M,\tau)$. The orthogonal projection $e_N$ onto $L^2(N,\tau)$ induces the unique $\tau$-preserving normal faithful conditional expectation $E_N:M\rightarrow N$ via
$$e_N\Lambda_\tau(x)=\Lambda_\tau(E_N(x)),\ \ \ x\in M.$$

The von Neumann subalgebra $M_1:=\la \pi_\tau(M),e_N\ra$ of $\cl B(L^2(M,\tau))$ generated by $\pi_\tau(M)$ and $e_N$ is the result of the basic construction of the inclusion $N\subseteq M$. The algebra $M_1$ has a canonical faithful semi-finite normal trace $\mathrm{tr}_1$ determined by $\mathrm{tr}_1(xe_Ny)=\tau(xy)$, $x,y\in M$ \cite[\S 1.1.2]{p94}. The trace $\tau$ is \textit{Markov} for the inclusion $N\subseteq M$ if $\mathrm{tr}_1$ is finite, and $\tau_1:=\mathrm{tr}_1(1)^{-1}\mathrm{tr}_1$ has $\tau_1|_M=\tau$. For example, if $M\cong\bigoplus_{i=1}^m M_{n_i}(\C)$ is finite dimensional and $N=\mathbb{C}1$, then the (unique, in this case) Markov trace for $\mathbb{C}1\subseteq M$ satisfies $\tau(p_{n_i})=\frac{n_i}{\mathrm{dim}(M)}$ for all minimal projections $p_{n_i}\in M_{n_i}(\C)$ and $i=1,...,m$ (see, e.g., \cite[Proposition 2.1]{ban99}). It follows that the trace on $M$ used in Theorem \ref{mainthm} is the Markov trace for the inclusion $\mathbb{C}1\subseteq M$.

A finite subset $B=\{\lm_1,...,\lm_n\}\subseteq M$ is a (right) \textit{Pimsner-Popa basis} for $M$ over $N$ if $x=\sum_{i=1}^n\lm_iE_N(\lm_i^*x)$, $x\in M$ \cite{pp86}. When $E_N(\lm_i^*\lm_j)=\delta_{i,j}1$, $\{\lm_i\}_{i=1}^n$ is \textit{orthonormal}. Pimsner-Popa bases play an important role when studying properties of the inclusion $N\subseteq M$ and its related mathematical structures, see, e.g., \cite{b94,jp11,pp86,p94,wat90} and references therein.

Inclusions which admit a unitary Pimsner-Popa basis have been studied \cite{bg20,bakshi2021few,cs06}. Motivated by a question of Popa \cite[\S 3.5]{p19}, Bakshi and Gupta recently showed that a finite-index regular inclusion $N\subseteq M$ of $II_1$-factors for which $N'\cap M$ is commutative or simple admits an orthonormal unitary Pimsner-Popa basis \cite[Theorem 3.21]{bakshi2021few}. They conjectured that any finite-index regular inclusion of $II_1$-factors admits such a basis \cite[Conjecture 3.20]{bakshi2021few}. Combining Theorem \ref{mainthm}, which generalizes \cite[Proposition 3.23]{bakshi2021few}, with the approach of Bakshi and Gupta, we now show that the conjecture holds.

\begin{theorem}\label{bg} Let $N\subseteq M$ be a finite-index regular inclusion of $II_1$-factors. Then $M$ admits an orthonormal unitary Pimsner-Popa basis.
\end{theorem}

\begin{proof}  By \cite[Lemma 3.24]{bakshi2021few} the restriction of $\tau$ to $N'\cap M$ is the Markov trace for the inclusion $\mathbb{C}1\subseteq N'\cap M$. Indeed, the proof of \cite[Theorem 3.12]{bg20} shows that the Watatani index of $\tau|_{N'\cap M}$ is scalar. If $N'\cap M=\bigoplus_{i=1}^m M_{n_i}(\C)$, one then considers the canonical embedding
$$N'\cap M \ni (x_1,...,x_m)\mapsto (1_{n_1}\otimes x_1)\oplus\cdots\oplus(1_{n_m}\otimes x_m)\in M_d(\C),$$  
where $d = \mathrm{dim}(N'\cap M)$ and sees readily, in view of \cite[Corollary 2.4.3]{wat90}
and \cite[Proposition 2.1]{ban99}, that the restriction of the normalized trace of $M_d(\C)$ agrees with $\tau|_{N'\cap M}$. Hence, Theorem \ref{mainthm} implies the existence of an orthonormal Pimsner-Popa basis $B:=\{u_i\}_{i=1}^n$ of $N'\cap M$ over $\mathbb{C}$ consisting of unitaries. The result then follows from the argument of Bakshi-Gupta, which we outline for the convenience of the reader (see the proof of \cite[Theorem 3.21]{bakshi2021few}). 

First,
\begin{equation*}
\begin{tikzcd}
&N'\cap M  \arrow[r,symbol=\subset] & N\vee (N'\cap M) \\
& \mathbb{C}\arrow[u,symbol=\subset]\arrow[r,symbol=\subset] & N\arrow[u,symbol=\subset]
\end{tikzcd}
\end{equation*}
is a non-degenerate commuting square (see \cite[\S 1.1.5]{p94}), where $N\vee (N'\cap M)$ is the von Neumann algebra generated by $N$ and $N'\cap M$. Thus, by \cite[Lemma 3.8]{bakshi2021few}, $B$ is also an orthonormal Pimsner-Popa basis for $N\vee (N'\cap M)$ over $N$. By regularity, it follows from \cite[Proposition 3.7]{bg20} that $M$ admits an orthonormal unitary Pimsner-Popa basis $\{v_j\}_{j=1}^m$ over $N\vee (N'\cap M)$. The product set $\{u_iv_j\}$ then forms a orthonormal unitary Pimsner-Popa basis of $M$ over $N$ (as shown in \cite[Theorem 3.21]{bakshi2021few}).
\end{proof}

An immediate corollary of Theorem \ref{bg} is an extension of \cite[Theorem 4.3]{bakshi2021few}.

\begin{corollary} Any finite-index regular inclusion $N\subseteq M$ of $II_1$-factors has depth at most 2.
\end{corollary}

\begin{proof} The argument of \cite[Theorem 4.3]{bakshi2021few} applies verbatim.
\end{proof}

The basis constructed in the proof of Theorem \ref{bg} lies in the normalizer of $N$ in $M$. Indeed, $\{u_i\}$ manifestly lies in $\mathcal{U}(N'\cap M)\subset\mathcal{N}_M(N)$, and by \cite[Proposition 3.7]{bg20}, $\{v_j\}$ also lies in $\mathcal{N}_M(N)$. Consequently, we can rephrase regularity of the inclusion $N\subseteq M$ as follows.

\begin{corollary}\label{c:bg} A finite-index inclusion $N\subseteq M$ of $II_1$-factors admits a (finite) orthonormal Pimsner-Popa basis in $\mathcal{N}_M(N)$ if and only if it is regular.
\end{corollary}

\begin{proof} If the inclusion is regular, then as noted above, the orthonormal basis constructed in the proof of Theorem \ref{bg} lies in the normalizer.

Conversely, if $M$ admits a (finite) Pimsner-Popa basis in $\mathcal{N}_M(N)$, then any $x\in M$ lies in the $N$-linear span of $\mathcal{N}_M(N)$, so the inclusion is regular.
\end{proof}

We remark that Pimsner-Popa bases of the form in Corollary \ref{c:bg} were recently shown in \cite{cckl} to produce unbiased teleportation schemes for tripartite quantum systems (in the commuting operator framework) arising from Jones' basic construction.

\vspace{0.1in}

{\noindent}{\it Acknowledgements.}
We are grateful to the referee for helpful comments. The first author was partially supported by the NSERC Discovery Grant RGPIN-2017-06275.  The second author was partially supported by the NSERC Discovery Grant RGPIN-2018-400160. The third author was partially supported by the NSERC Discovery Grant RGPIN-2022-04149.

\bibliographystyle{plain}

\bibliography{orthun}

\end{document}